\documentclass[11pt]{amsart}

\usepackage{amssymb,amsfonts,psfrag,amscd,enumerate,epsfig,latexsym,amsmath}
\usepackage{subfigure}
\usepackage[colorlinks,urlcolor=green,citecolor=blue,linkcolor=blue]{hyperref}

 \pdfoutput=1
 
\usepackage{xspace}

\newtheorem{thm}{Theorem}
\newtheorem{lem}[thm]{Lemma}
\newtheorem{cor}[thm]{Corollary}

\theoremstyle{definition}

\newtheorem{defn}[thm]{Definition}

\newtheorem{example}[thm]{Example}

\theoremstyle{remark}

\newtheorem{rmk}[thm]{Remark}

\author{Valent\'{\i}n Mendoza}
\address{Departamento de Matem\'atica, Universidade Federal de Vi\c cosa, MG, Brazil.}
\email{valentin@ufv.br}
\title{Renormalization and forcing of Horseshoe Orbits}



\date{May 20, 2014.}
\begin{document}
\begin{abstract}
In this paper we deal with the Boyland order of horseshoe orbits. We prove 
that there exists a set $\mathcal{R}$ of renormalizable horseshoe orbits 
containing only quasi-one-dimensional orbits, that is, for these orbits 
the Boyland order coincides with the unimodal order.
\end{abstract}

\maketitle
\bibliographystyle{amsalpha}

\section{Introduction}
In  \cite{Boy}, Boyland introduced the forcing relation between periodic orbits 
of the disk $D^2$. Given two periodic orbits $P$ and $R$, we say that
 $P$ \textit{forces} $R$, denoted by $P\geqslant_2R$, if  every homeomorphism
  of $D^2$ containing the braid  type of $P$ must contain the braid type of $R$.
   The set of periodic orbits forced by $P$
 is denoted by $\Sigma_P$. In this paper we are concerned with the
 forcing of Smale horseshoe periodic orbits. A horseshoe orbit $P$ is called 
 \textit{quasi-one-dimensional} if $P$ forces all orbit $R$ such that $P\geqslant_1R$,
 where $\geqslant_1$ is the unimodal order. In \cite{Hall}, Hall gave a set of 
quasi-one-dimensional horseshoe orbits, called NBT orbits (Non-Bogus 
Transition orbits) which are in bijection with $\mathbb{Q}\cap(0,\frac{1}{2})$
 and have the property that their thick interval map induced has minimal periodic 
 orbit structure, that is, if $P$ is an NBT orbit then every braid type of a
  periodic orbit of its thick interval map $\theta_P$ is forced by the braid type
   of $P$.

In this paper we obtain a type of orbits which are quasi-one-dimensional too
 although  their associated thick interval maps are reducible in the sense of 
Thurston \cite{Thu}, that is, they are isotopic to reducible homeomorphisms which 
have an invariant set of non-homotopically trivial disjoint curves $\{C_1,\cdots,C_n\}$.
Restricted to the components of $D^2\setminus\{C_1,\cdots,C_n\}$, these reducible maps
(or one of its power) have minimal periodic orbit structure.
\begin{thm}\label{maintheorem}
There exists a set $\mathcal{R}\supset\operatorname{NBT}$ of quasi-one-dimensional horseshoe orbits,
 that is, if $P\in\mathcal{R}$ then $\Sigma_P=\{R:P\geqslant_1R\}$.
\end{thm}

These orbits are defined using the renormalization operator which was introduced 
in \cite{ColEck} as the $*$-product.

\section{preliminaries}
\subsection{Boyland partial Order}
 Let $D_n$ be the punctured disk. Let $\operatorname{MCG}(D_n)$ be the 
 group of isotopy classes of homeomorphisms of $D_n$, which is called the 
 \textit{mapping class group of $D_n$}. Given a homeomorphism $f:D^2\rightarrow D^2$
 of the disk $D^2$ with a periodic orbit $P$, the \textit{braid type}
 of $P$, denoted by $\operatorname{bt}(P,f)$, is defined as follows:
 Take an orientation preserving  homeomorphism $h:D^2\setminus P\rightarrow D_n$
 then  $\operatorname{bt}(P,f)$ is the conjugacy class 
 $[h\circ f\circ h^{-1}]\in \operatorname{MCG}(D_n)$ of $h\circ f\circ h^{-1}:D_n\rightarrow D_n$.
 
 Let $\operatorname{BT}$ be the union of all the periodic braid types  and let $\operatorname{bt}(f)$
  be the  set formed by the braid types  of the periodic orbits of $f$. We will say that
 $f:D^2\rightarrow D^2$ \textit{exhibits} a braid type  $\beta$ if there exists
  an $n$-periodic orbit $P$  for $f$ with $\beta=\operatorname{bt}(P,f)$. Now we
  can define the relation $\geqslant_2$ on $\operatorname{BT}$. We say that $\beta_1$ \textit{forces}
  $\beta_2$, denoted by $\beta_1\geqslant_2\beta_2$, if every homeomorphism
  exhibiting $\beta_1$, exhibits $\beta_2$ too. Then it is said that 
  a periodic orbit $P$ \textit{forces} another periodic orbit $R$,
  denoted by $P\geqslant_2 R$, if 
  $\operatorname{bt}(P)\geqslant_2\operatorname{bt}(R)$.
 
 In \cite{Boy}, P. Boyland proved the following theorem.
 \begin{thm}{\cite[Theorem 9.1]{Boy}}\label{thm:boy}
 The relation $\geqslant_2$ is a partial order.
 \end{thm}
\subsection{Smale horseshoe}
The Smale horseshoe is a map $F:D^2\rightarrow D^2$ of the disk which acts as
in Fig. \ref{figurahorseshoe}. The set $\Omega=\cap_{j\in\mathbb{Z}} F^j(V_0\cup V_1)$
is $F$-invariant and $F|_{\Omega}$ is conjugated to the shift $\sigma$ on the 
sequence space of two symbols $0$ and $1$, $\Sigma_2=\{0,1\}^\mathbb{Z}$, where
\begin{equation}
\sigma((s_i)_{i\in\mathbb{Z}})=(s_{i+1})_{i\in\mathbb{Z}}.
\end{equation}
The conjugacy 
$h:\Omega\rightarrow\Sigma_2$ is defined by
\begin{equation}
(h(x))_i=\left\{ \begin{array}{cc}
0 & \textrm{ if }F^i(x)\in V_0, \\
1 & \textrm{ if } F^i(x)\in V_1.\\
\end{array}
\right.
\end{equation}
\begin{figure}[h]
\centering
\includegraphics[width=55mm,height=36mm]{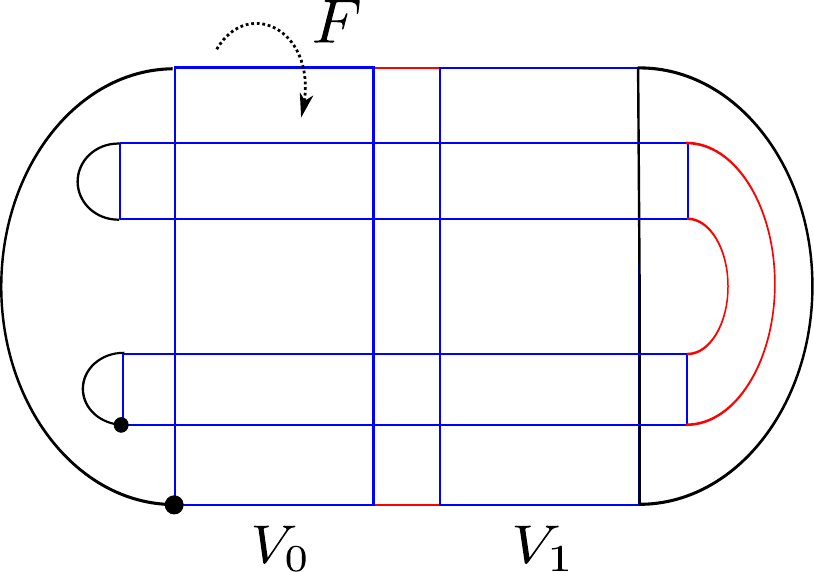}
\caption{Dynamics of $F$.}
\label{figurahorseshoe}
\end{figure}

To compare horseshoe orbits  it is necessary to define the \textit{unimodal order}.
 It is a total order in $\Sigma^+=\{0,1\}^\mathbb{N}$ given by the following rule:
 Let $s=s_0s_1...$ and $t=t_0t_1...$  be sequences in $\Sigma^+$ such that 
 $s_i=t_i$ for $i\le k$ and $s_{k+1}\neq t_{k+1}$, then $s<t$ if
 \begin{itemize}
 \item[(O1)] $\sum _{i=0}^{k} s_i$ is even and $s_{k+1}<t_{k+1}$, or
 \item[(O2)] $\sum _{i=0}^{k} s_i$ is odd and $s_{k+1}>t_{k+1}$.
 \end{itemize}
 We say that $s\geqslant_1 t$ if either $s=t$ or $s>t$.

Every $n$-periodic orbit $P\in\Omega$ of $F$ has a \textit{code} denoted by  $c_P\in\Sigma_2$. 
It is obtained from $h(p)=c_P^\infty$ where $p$ is a point of $P$  and $c_P$ satisfies
 $\sigma^i(c_P)\leqslant_1 c_P$, that is, $c_P$ is \textit{maximal} in the unimodal
  order $\geqslant_1$.
 We say that $P\geqslant_1 R$ if $\sigma^n(R)\leqslant_1 c_P$, $\forall n\ge1$.
 For every orbit $P$, there exists a homeomorphism $\theta_P$ that realizes the
 combinatorics of $P$. This is obtained fatting the line diagram of $P$ and 
 it is called the \textit{tick map induced by $P$}. See \cite{Hall}.
\subsection{Renormalized Horseshoe Orbits}

Let $P$ and $Q$ be two horseshoe periodic orbits with codes 
$c_P=A a_{n-1}$ where $A=a_0a_1\cdots a_{n-2}$ and $c_Q=b_0b_1\cdots b_{m-2} b_{m-1}$ 
with periods $n$ and $m$, respectively.
\begin{defn}[Renormalization Operator]
We will write $P*Q$ for the $nm$-periodic orbit with code
\begin{equation}
c_{P*Q}=\left\{ \begin{array}{cc}
Ab_0Ab_1Ab_2\cdots Ab_{m-2}Ab_{m-1} & \textrm{ if }\epsilon(A) \textrm{ is even} \\
A\overline{b_0}A\overline{b_1}A\overline{b_2}\cdots A\overline{b_{m-2}}A\overline{b_{m-1}} & \textrm{ if }\epsilon(A) \textrm{ is odd} \\
\end{array}
\right.
\end{equation}
where $\epsilon(A)=\sum_{i=0}^{n-2}a_i$ and $\overline{b_i}=1-b_i$.
\end{defn}
The orbit $P*Q$ is called the \textit{renormalization of  $P$ and $Q$}. 
If an orbit $S$ satisfies $S=P*Q$ for some $P,Q\in\Sigma_2$, it is said that $S$
is \textit{renormalizable}. Also we will denote
$$
P_1*P_2*\cdots*P_k=(\cdots((P_1*P_2)*P_3)\cdots).
$$

\begin{example}
If $P=101$ and $Q=1001$ then $P*Q=10\mathbf{0}10\mathbf{1}10\mathbf{1}10\mathbf{0}$.
\end{example}
\subsection{NBT Orbits}
There are a type of horseshoe orbits for which the Boyland partial order is well-understood.
They are constructed in the following way. Given a rational number
 $q=\frac{m}{n}\in\widehat{\mathbb{Q}}:=\mathbb{Q}\cap(0,\frac{1}{2})$, let $L_q$ be the straight line segment
joining $(0,0)$ and $(n,m)$ in $\mathbb{R}^2$. Then construct a finite word $c_q=s_0s_1\cdots s_{n}$
as follows:
\begin{equation}
s_i=\left\{ \begin{array}{cc}
1 & \textrm{ if $L_q$ intersects some line $y=k, k\in\mathbb{Z}$, for $x\in(i-1,i+1)$} \\
0 & \textrm{ otherwise }\\
\end{array}
\right.
\end{equation}
It follows that $c_q$ is palindromic and has the form:
\begin{equation}
c_q=10^{\mu_1}1^20^{\mu_2}1^2\cdots1^20^{\mu_{m-1}}1^20^{\mu_m}1.
\end{equation}
We will denote $P_q$ to the periodic orbits of period $n+2$ which have
 the codes $c_q{}_0^1$, when the distinction is not important and let 
 $\operatorname{NBT}=\{P_q:q\in\widehat{\mathbb{Q}}\}$. In \cite{Hall},
 Hall proved the following result.
 \begin{thm}
Let $q,q'\in\widehat{\mathbb{Q}}$. Then
\begin{itemize}
\item[(i)]  $P_q$ is quasi-one-dimensional, that is, $P_q\geqslant_1R \Longrightarrow P_q\geqslant_2 R$.
\item[(ii)] $q\leqslant q' \Longleftrightarrow (c_{q}{}^0_1)^\infty\geqslant_1 (c_{q'}{}^0_1)^\infty \Longleftrightarrow (c_{q}{}^0_1)^\infty\geqslant_2 (c_{q'}{}^0_1)^\infty $
\end{itemize}
 \end{thm}
So theorem above says that the Boyland order restricted to the NBT orbits is 
equal to the unimodal order.

\section{Forcing of renormalizable orbits}
For proving Theorem \ref{maintheorem} we need the following result:
\begin{thm}\label{theoforcing}
Let $P=A_1^0$ and $Q$  be periodic orbits.
Then
\begin{equation}
\Sigma_{P*Q}=\{R:P\geqslant_2 R\}\cup\{P*R:Q\geqslant_2R\}.
\end{equation}
\end{thm}
To prove the result above it will be needed two lemmas whose proofs are left to the 
reader.
\begin{lem}\label{lemma:5}
Let $i,j\in\{1,\cdots,n-1\}$ be positive integers with $i\neq j$ and $T_M=P*Q$
and $T_m=\sigma^n(T_M)$. Then
\begin{itemize}
\item[(a)] if $\epsilon(A)$ is even then $A0^\infty\leqslant_1 T_m\leqslant_1 T_M \leqslant_1 A1^\infty$,
\item[(b)] if $\epsilon(A)$ is odd then $A1^\infty\leqslant_1 T_m\leqslant_1 T_M \leqslant_1 A0^\infty$,
\item[(c)] $\sigma^i(P)\leqslant_1 \sigma^j(P) \Longleftrightarrow [\sigma^i(T_M)\leqslant_1 \sigma^j(T_M)
\textrm{ and } \sigma^i(T_m)\leqslant_1 \sigma^j(T_m)]$.
\end{itemize}
\end{lem}

\begin{lem}\label{lemma:6}
Let $P=A{}_1^0$ and $Q$ be two periodic orbits. If $i,j\in\{0,\cdots,m-1\}$, with $i\neq j$, then
 $$\sigma^i(Q)>_1\sigma^j(Q) \Longleftrightarrow \sigma^{in}(P*Q)>_1\sigma^{jn}(P*Q).$$
\end{lem}

\begin{proof}[Proof of Theorem \ref{theoforcing}]
Let $\theta_{P*Q}$ be the thick map induced by $P*Q$. 
First we see that the only iterates of $P*Q$ satisfying 
$T_m\leqslant_1 \sigma^i(P*Q)\leqslant_1 T_M$ are
the iterates $\sigma^{in}(P*Q)$, with $0\le i\le m-1$; so there exists a 
curve $C_{n-1}$ containing these orbits disjoint from the others and bounding a 
region $D_{n-1}$.
 By Lemma \ref{lemma:5}(c) and noting that $\sigma^i(T_m)$ and
 $\sigma^i(T_M)$ has the same initial symbol for $i\in\{1,\cdots,n-1\}$, 
 it follows that $\{\theta_{P*Q}^i(C_{n-1})\}_{i=1}^{n-1}$ has the same combinatorics as $P$. 
 For $i=0,\cdots,n-2$, let $C_i=\theta_{P*Q}^{i+1}(C_{n-1})$ be a curve 
 which bounds a domain $D_i$. It is possible to define $\theta_{P*Q}$ 
 such that  $\theta_{P*Q}^n(C_{n-1})=C_{n-1}$.  Then the line diagram 
 of $\{D_0,\cdots, D_{n-1}\}$  is as the line
 diagram of $P$ and then $\theta_{P*Q}$ has the same behaviour than 
 $\theta_P$ in the exterior of $\cup D_i$. Since $\theta_{P*Q}$ can be reduced
  by a family of curves, we will need study the Thurston representative 
of $\theta_{P*Q}$ restricted to $D^2\setminus\cup C_i$. As $\theta_P$ and $\theta_{P*Q}$
 have the same  combinatorics in the exterior of $\cup D_i$, they have the same
  Thurston representative in the exterior of $\cup D_i$. So $P$ and $P*Q$ 
  force the same periodic orbits in the exterior of $\cup D_i$.
  Then $\{R:P\geqslant_2 R\}\subset\Sigma_{P*Q}$.
  See Fig. \ref{figorbitmaximal}.
\begin{figure}[h]
\centering
\includegraphics[width=85mm,height=16mm]{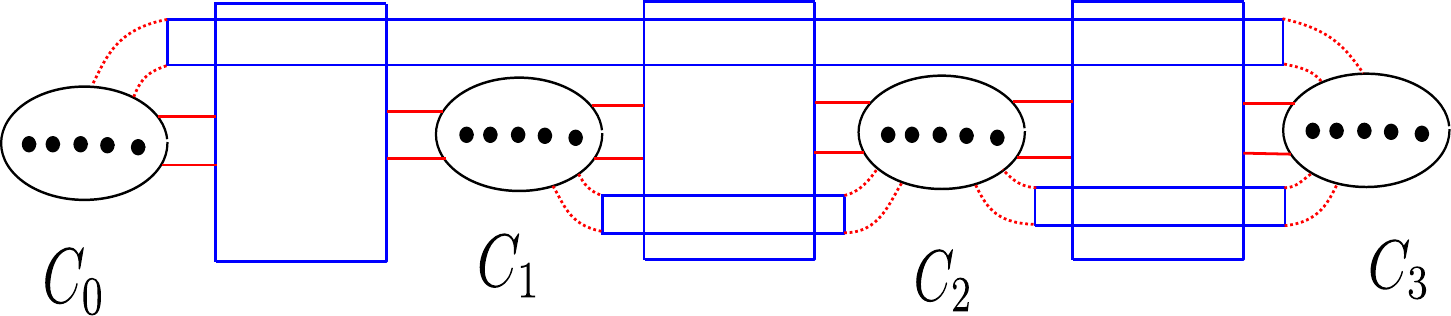}
\caption{The image $\theta_{P*Q}(C_i)$ when  $P=1001$.}
\label{figorbitmaximal}
\end{figure}

It is clear that to find what orbits are forced by $P*Q$ in $\cup D_i$, it is 
enough to study $\theta_{P*Q}^n$ restricted to $D_{n-1}$. By Lemma \ref{lemma:6},
the line diagram of $\theta_{P*Q}^n$ inside $D_{n-1}$ is the same as the line diagram
of $Q$ when $\epsilon(A)$ is even, and it is flipped when $\epsilon(A)$ is odd.
See Fig. \ref{fig:Qdiagramline}.  So $\theta_{P*Q}^n$ has the same combinatorics
than $\theta_Q$.
\begin{figure}[h]
\centering
\subfigure[When $\epsilon(A)$ is even and $Q=10010$.]{
\includegraphics[width=70mm,height=35mm]{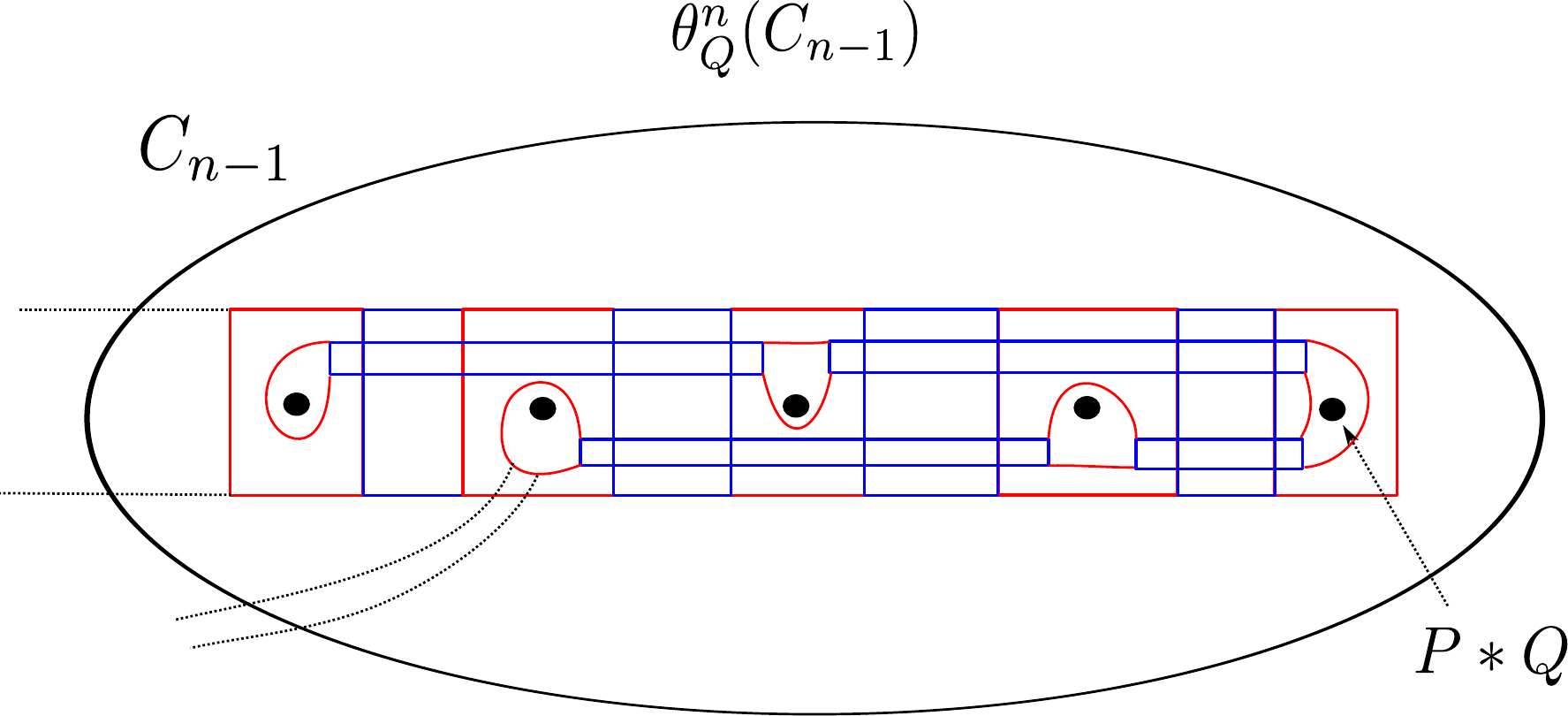}
\label{subfig2a}
}
\hspace{8pt}
\subfigure[When $\epsilon(A)$ is odd and $Q=10010$.]{
\includegraphics[width=70mm,height=35mm]{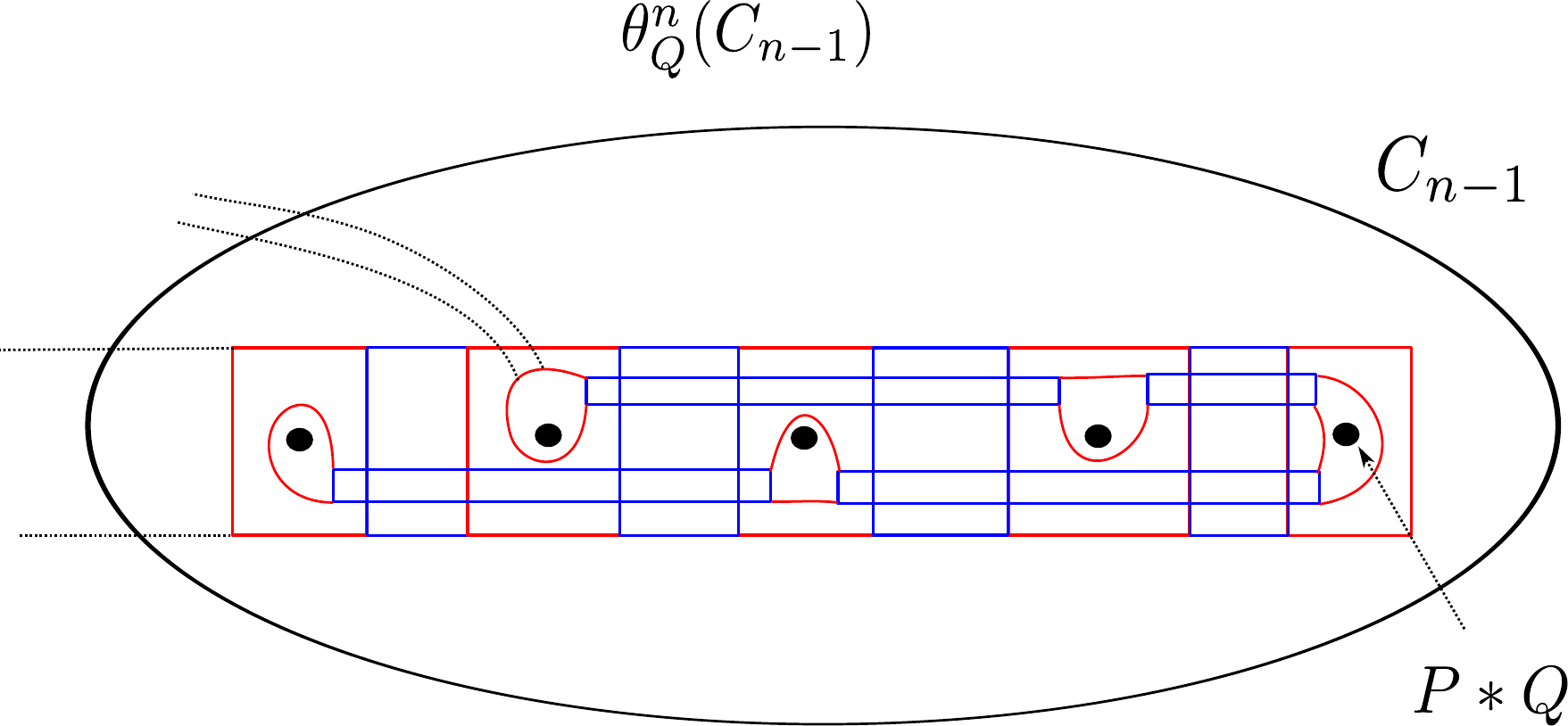}
\label{subfig2b}
}
\label{fig:Qdiagramline}
\caption{The image $\theta_{P*Q}^{n}(C_{n-1})$.}\label{fig:maximalorbita}
\end{figure}

As in Lemma \ref{lemma:6}, we can prove that 
$R\leqslant_2 Q \Longleftrightarrow P*R\leqslant_2 P*Q.$
So $\Sigma_{P*Q}=\{R:P\geqslant_2 R\}\cup\{P*R:Q\geqslant_2R\}$.
\end{proof}

\begin{rmk}
Theorem \ref{theoforcing} says us that to look for the orbits that are forced
by $P*Q$ it is enough to look for the orbits that are forced by $P$ and the
orbits that are forced by $Q$. So we can study the thick maps induced
by $P$ and $Q$ separately. Every of these thick maps can be reduced using 
methods to determine its minimal representative, e.g. \cite{BesHan, SolNat}.
\end{rmk}

\begin{cor}\label{corone}
Let $P_1,P_2,\cdots,P_k$ be NBT orbits. Then
\begin{itemize}
\item[(a)] $\Sigma_{P_1*\cdots*P_k}=\cup_{j=1}^k \{P_1*\cdots*P_{j-1}*R:P_j\geqslant_1 R\}$
\item[(b)] $\Sigma_{P_1*\cdots*P_k}=\{R:P_1*\cdots*P_k\geqslant_1R\}$
\end{itemize}
\end{cor}
\begin{proof}
Item (a) follows directly from Theorem \ref{theoforcing}. For item (b) it is enough
to prove that if $P$ and $Q$ are quasi-one-dimensional horseshoe orbits then 
$P*Q$ is a quasi-one-dimensional orbit too. Suppose that $\epsilon(A)$ is even. 
From Theorem \ref{theoforcing},
\begin{equation}
\Sigma_{P*Q}=\{R:P\geqslant_1 R\}\cup\{P*R:Q\geqslant_1R\}.
\end{equation}
Hence it follows that $\Sigma_{P*Q}\subset\{S:P*Q\geqslant_1 S\}$.
We have to prove the inclusion $\{S:P*Q\geqslant_1 S\}\subset\Sigma_{P*Q}$.
If $P\geqslant_1 S$ then $S\in\Sigma_{P*Q}$. 
Let $S$ with $c_S=s_0s_1\cdots s_{k-1}$
be a periodic orbit with $P\leqslant_1S\leqslant_1 P*Q$. 
By Lemma \ref{lemma:5}(a), $(A0){}^\infty \leqslant_1 c_S^\infty\leqslant_1 (A1){}^\infty$.
This implies that $c_S=As_{n-1}s_n\cdots$ and
$$(0A)^\infty \leqslant_1\sigma^{n-1}(c_S^\infty)=s_{n-1}s_n\cdots \leqslant_1(1A)^\infty,$$
and $\sigma^n(c_S{}^\infty)\geqslant_1(A0)^\infty$. In the other hand 
$\sigma^{n}(c_S{}^\infty)\leqslant_1 c_{P*Q}{}^\infty$.
Then $\sigma^n(c_S{}^\infty)=As_{2n-1}\cdots$ and then 
$c_S=As_nAs_{2n-1}\cdots$.  Continuing this process, it follows that
$S=P*R$ where $c_R=s_{n-1}s_{2n-1}\cdots$. So $P*R\leqslant_1 P*Q$ 
which implies that $R\leqslant_1 Q$. So $S\in\{P*R:Q\geqslant_1R\}$ and the 
proof is finished.

\end{proof}


Now we proceed to prove Theorem \ref{maintheorem}.
\begin{proof}[Proof of Theorem \ref{maintheorem}]
Let $\{P_j\}_{j\in\mathbb{N}}$ be the set of NBT orbits and consider the
space $\mathbb{N}^\mathbb{N}$ of sequences of positive integers. 
Take a sequence $\mathcal{J}=(j_1,j_2,\cdots,j_n,\cdots)\in\mathbb{N}^\mathbb{N}$
and define 
\begin{equation}
\mathcal{R}_\mathcal{J}=\cup_{k=1}^\infty\{P_{j_1}*\cdots*P_{j_k}\}
\end{equation}
and $\mathcal{R}=\cup_{\mathcal{J}\in\mathbb{N}^\mathbb{N}}\mathcal{R}_{\mathcal{J}}$.
By Corollary \ref{corone}(b), every orbit of $\mathcal{R}$ is quasi-one-dimensional.
\end{proof}

 \begin{example}
By previous Theorem, 
if $S=10010*10010=1001\mathbf{1}1001\mathbf{0}1001\mathbf{0}1001\mathbf{1}1001\mathbf{0}$, 
\begin{equation}
\Sigma_S=\{R:10010\geqslant_1R\}\cup\{10010*R:10010\geqslant_1R\}.
\end{equation}
\end{example}





\section{acknowledgements}
I would like to thank to Toby Hall and Dylene de Barros for their comments which contributed to 
improve this paper.

\end{document}